\documentclass[a4 paper, 12 pt]{article}
\usepackage{amsmath,amsthm,amsfonts,latexsym,amsopn,verbatim,amscd,amssymb,enumerate,amssymb}
\usepackage{geometry}
\newgeometry{a4paper,width=150mm,left=3cm,right=2.5cm,top=4cm,bottom=4cm}
\usepackage{blindtext}
\usepackage{mathtools}
\usepackage{mathrsfs}
\DeclarePairedDelimiter\floor{\lfloor}{\rfloor}
\usepackage{sectsty}
\usepackage{graphicx}
\usepackage[square,numbers]{natbib}
\usepackage[colorlinks=true,urlcolor=blue, citecolor=blue]{hyperref}
\usepackage{footnote}
\usepackage{doi}
\usepackage{tikz-cd}
\usepackage[utf8]{inputenc}
\usepackage[english]{babel}
\usepackage[document]{ragged2e}
\theoremstyle{plain}
\newtheorem{theorem}{Theorem}[section]

\newtheorem{corollary}[theorem]{Corollary}
\newtheorem{proposition}[theorem]{Proposition}

\theoremstyle{definition}
\newtheorem{definition}[theorem]{Definition}

\newtheorem{example}[theorem]{Example}

\theoremstyle{remark}

\begin{document}
\begin{center}
\large{On semicommutativity of  rings relative to hypercenter}
\end{center}
Nazeer  Ansari\footnote[1]{Corresponding author}\\
Department of Mathematics\\
Madanapalle Institute of Technology $\&$ Science\\
Madanapalle, Andhra  pradesh, 517325\\
email:drnazeeransari@mits.ac.in\\
\vspace*{.5 cm}

Kh. Herachandra singh\\
Department of Mathematics\\
Manipur University, Canchipur\\
Imphal, Manipur, 795003\\
 email: heramath@manipuruniv.ac.in\\
 \vspace*{.5 cm}
\justify
\begin{abstract}
Armendariz and semicommutative rings are generalizations of reduced rings. In \cite{IN}, I.N. Herstein introduced the notion of a hypercenter of a ring to generalize the center subclass. For a ring $R$, an element $a \in R$ is called hypercentral if $ax^{n}=x^{n}a$ for all $x \in R$ and for some $n=n(x,a) \in \mathbb{N}$. Motivated by this definition, we introduce $\mathscr{H}$-Semicommutative rings as a generalization of semicommutative rings and investigate their relations with other classes of rings. We have proven that the class of $\mathscr{H}$-Semicommutative rings lies strictly between Zero-Insertive rings (ZI) and Abelian rings. Additionally, we have demonstrated that if $R$ is $\mathscr{H}$-semicommutative, then for any $n \in \mathbb{N}$, the matrix subring $S_{n}^{'}(R)$ is also $\mathscr{H}$-semicommutative. Among other significant results, we have established that if $R$ is $\mathscr{H}$-semicommutative and left $SF$, then $R$ is strongly regular. We have also shown that $\mathscr{H}$-semicommutative rings are 2-primal, providing sufficient conditions for a ring $R$ to be nil-singular. Additionally, we have proven that if every simple singular module over $R$ is wnil-injective and $R$ is $\mathscr{H}$-semicommutative, then $R$ is reduced. Furthermore, we have studied the relationship of $\mathscr{H}$-semicommutative rings with the classes of Baer, Quasi-Baer, p.p. rings, and p.q. rings in this article, and we have provided some more relevant results.

\end{abstract}
{\bf keywords:} Hypercenter; semicommutative ring; $\mathscr{H}$-Semicommutative ring.

\section{Introduction} 

In this article, $R$ represents an associative ring with unity (unless $R$ is a nil ring). We write $rad(R)$, $Nil(R)$, $Nil_{\star}(R)$, $Nil^{\star}(R)$, $Z(R)$, and $T(R)$ to represent the Jacobson radical, the set of all nilpotent elements, the lower nilradical, the upper nilradical, the set of all central elements, and the set of all hypercentral elements of $R$, respectively. For a non-empty set $X\subseteq R$, an element $a \in R$ is called left(right)-annihilator if $aX=0(Xa=0)$. We also write $l_{R}(X)(R_{R}(X))$ to denote the set of all left(right)-annihilators for a non-empty set $X\subseteq R$. A ring $R$ is called {\it reduced} if $Nil(R) = (0)$. A ring $R$ is reversible if, for any $u, v \in R$, the condition $uv = 0$ implies $vu = 0$. A ring \( R \) is called {\it semicommutative } if, whenever \( u, v \in R \) satisfy \( uv = 0 \), the condition \( uRv = 0 \) holds.  {\it semicommutative} rings are a generalization of {\it reduced} rings. $R$ is called {\it directly finite} if for any $u, v \in R$, $uv = 1$ implies $vu = 1$. $R$ is called an $NI$-ring if $Nil(R)$ forms an ideal of $R$. Many authors have extensively studied {\it semicommutative} rings for the past few decades and developed many new generalizations. Hwang, in \cite{HH}, introduced the notion of an $R$-{\it IFP (insertion of factor property)} ring. $R$ is called $R$-{\it IFP} if, for any $u, v \in R$ with $v \neq 0$, $uv = 0$ implies $uRw = 0$ for some $0 \neq w \in R$. Chen, in \cite{CC}, introduced {\it weakly semicommutative} rings as a generalization of {\it semicommutative} rings. $R$ is called {\it weakly semicommutative} if the condition $uRv \in Nil(R)$ holds true whenever $u, v \in R$ satisfy $uv = 0$. In \cite{YJ}, $R$ is called {\it nil-semicommutative} if whenever $uv \in Nil(R)$ for some $u, v \in R$, then $uRv \subseteq Nil(R)$. We denote these rings as {\it nil-semicommutative-II} to avoid nomenclature clashes with other generalizations. In \cite{RAM},  $R$ is called {\it nil-semicommutative}  if any elements $u$ and $v$ in $Nil(R)$ that satisfy $uv = 0$ also satisfy $uRv = 0$. We denote these rings by {\it nil-semicommutative-I}. In \cite{TNA},  $R$ is called {\it central semicommutative} ring if the condition   $uRv \subseteq Z(R)$ holds whenever  $u,v \in R$ satisfy $uv=0$. In \cite{XX},  $R$ is called {\it J-semicommutative} ring if the condition   $uRv \subseteq rad(R)$ holds whenever  $u,v \in R$ satisfy $uv=0$.
In \cite{IN}, I.N. Herstein introduced the notion of the hypercenter of a ring. For a ring $R$, an element $a \in R$ is called hypercentral if $ax^{n}=x^{n}a$ for every $x \in R$ and for some $n=n(x, a) \in \mathbb{N}$. Thus, the hypercenter of ring $R$ is defined as:
\begin{center}
$T(R)=\{a \in R \,|\, ax^{n}=x^{n}a \, \forall \, x\in R,  \, n=n(x, a)\geq 1\}$
\end{center}
From the definition, it is easy to see that the hypercenter forms a subring of $R$ and that $Z(R) \subseteq T(R)$. For a ring $R$, we note that $T(R)$ does not necessarily equal $Z(R)$. For instance, if $R$ is a noncommutative nil ring, then $T(R) = R$ but $Z(R) \neq R$. Thus, we can see the hypercenter as a generalization of the center of a ring. However, I.N. Herstein in \cite{IN} showed that if $R$ is either semiprime or has no non-zero nil ideals, then the hypocenter and the centre of $R$ coincide. Given these generalizations of $Z(R)$, a natural question arises: what happens if we replace $Z(R)$ with $T(R)$ in the class of generalized semicommutative rings?

The purpose of this article is to introduce $\mathscr{H}$-{\it semicommutative} rings as a generalization of {\it semicommutative} rings and investigate their relations with other classes of rings. We say a ring $R$ is $\mathscr{H}$-{\it semicommutative} if whenever $ab=0$ for some $a, b \in R$, then $arb$ is hypercentral for each $r \in R$. We have proven that the class of $\mathscr{H}$-Semicommutative rings lies strictly between Zero-Insertive rings (ZI) and Abelian rings. Additionally, we have shown that if $R$ is $\mathscr{H}$-semicommutative, then for any $n \in \mathbb{N}$, the matrix subring $S_{n}^{'}(R)$ is also $\mathscr{H}$-semicommutative. Among other important results, we have demonstrated that if $R$ is $\mathscr{H}$-semicommutative and left $SF$, then $R$ is strongly regular. We show that $\mathscr{H}$-semicommutative rings are 2-primal and provide sufficient conditions for a ring $R$ to be nil-singular. Furthermore, we prove that if every simple singular module over $R$ is wnil-injective and $R$ is $\mathscr{H}$-semicommutative, then $R$ is reduced. Additionally, we study the relationship of $\mathscr{H}$-semicommutative rings with the classes of Baer, Quasi-Baer, p.p. rings, and p.q. rings and present more relevant results.

\section{$\mathscr{H}$-semicommutative rings} 
Here, we introduce $\mathscr{H}$-{\it semicommutative} rings, which contain the class of {\it semicommutative} rings. On the other hand, every reduced ring is Armendariz and {\it semicommutative}. We demonstrate that there exists a large class of $\mathscr{H}$-{\it semicommutative} rings that are not reduced. Thus, $\mathscr{H}$-{\it semicommutative} rings constitute an independent category within the class of {\it semicommutative} rings.

\begin{definition}\label{d1}
A ring $R$ is called $\mathscr{H}$-{\it semicommutative} if $uv=0$ implies that $uRv \subseteq T(R)$ for any $u, v \in R$.
\end{definition}
From the above definition, we can easily say that the class of $\mathscr{H}$-semicommutative rings is closed under subrings. Furthermore, all reduced and central semicommutative rings are $\mathscr{H}$-semicommutative rings. Next, we have noted a result that generates a large class of rings which are $\mathscr{H}$-semicommutative but neither reduced nor central-semicommutative.
\begin{proposition}\label{p3}
Every nil rings are $\mathscr{H}$-{\it semicommutative}.
\end{proposition}
\begin{proof}
The proof follows easily since $T(R)=R$.
\end{proof}
\begin{proposition}\label{p2.3}
Let $R$ be $\mathscr{H}$-{\it semicommutative}. Then $R$ is central semicommutative if any of the following conditions is true.
\begin{itemize}
\item[(1)] $R$ is division ring.
\item[(2)] $R$ is semisimple ring.
\item[(3)] $R$ be a ring with no nonzero nil ideal.
\end{itemize}
\end{proposition}
\begin{proof}
The proof is follows easily from \cite{IN}.
\end{proof}
\begin{proposition}\label{p1}
Every $\mathscr{H}$-{\it semicommutative} ring is nil-semicommutative-II.
\end{proposition}
\begin{proof}
\justify
Let $u \in N(R)$ such that $u^{n}=0$. Thus by Lemma 2.2 in \cite{LJ}, it is enough to show that $ru \in N(R)$ for all $r \in R$. We have $u^{n}=u^{n-1}.u=0 \Rightarrow u^{n-1}ru \in T(R)$ for all $r \in R$. This implies $(u^{n-1}ru)(ru)^{m}=(ru)^{m}(u^{n-1}ru)=(ru)^{m-1}ru(u^{n-1}ru)=0$.  Thus we get $u^{n-1}(ru)^{m+1}=0$. Hence $u^{n-2}ru(ru)^{m+1} \in T(R)$. This implies $u^{n-2}ru(ru)^{m+1}(ru)^{p}=(ru)^{p}(u^{n-2}ru(ru)^{m+1})=(ru)^{p-1}.r(u^{n-1}ru(ru)^{m+1})=0$. Thus we get $u^{n-2}ru(ru)^{m+1}(ru)^{p}=u^{n-3}u(ru)^{m+p+2}=0$. Hence $u^{n-3}ru(ru)^{m+p+2} \in T(R)$.  Proceeding in similar way we get $u^{n-3}(ru)(ru)^{m+p+2}(ru)^{q}=(ru)^{q}\{u^{n-3}ru(ru)^{m+p+2}\}=(ru)^{q-1}r\{u^{n-2}ru(ru)^{m+1}(ru)^{p}\}=0$. Hence we get $u^{n-3}(ru)^{m+p+q+3}=0$. Proceeding similarly we get $(ru)^{t}=0$ for some $t\in \mathbb{N}$. This implies $ru \in N(R)$. 
\end{proof}

\begin{proposition}
If a ring $R$ is $\mathscr{H}$-semicommutative, then $R$ is abelian and for any idempotent $e=e^{2} \in R$ the rings $eR$ and $(1-e)R$ are $\mathscr{H}$-semicommutative.
\end{proposition}
\begin{proof}
Since $eR$ and   $(1-e)R$ are subring of $R$ recall that $\mathscr{H}$-{\it semicommutative} rings are closed under subring. Thus now it is left to  show that $R$ is abelian. We have  $e(1-e)=0$ which implies $er(1-e) \in T(R)$.  This implies   for some $n \in \mathbb{N}$, we have $er(1-e)(1-e)^{n}=(1-e)^{n}er(1-e)=(1-e)er(1-e)=0$. Thus we have $er(1-e)^{n+1}=er(1-e)=0\Rightarrow er=ere$. Again  $(1-e)e=0$ implies $(1-e)re \in T(R)$ for all $r \in R$. Processing similarly we have for some $n \mathbb{N}$,  $(1-e)^{n}(1-e)re=(1-e)re(1-e)^{n}=0$. This implies $(1-e)^{n+1}re=(1-e)re=0 \Rightarrow re=ere$. Thus we get $er=re$. Hence $R$ is abelian.
\end{proof}
\begin{proposition}\label{p2}
If a ring $R$ is abelian and for some idempotent $e=e^{2} \in R$ the rings $eR$ and $(1-e)R$ are $\mathscr{H}$-semicommutative, then $R$ is $\mathscr{H}$-semicommutative.
\end{proposition}
\begin{proof}
 let us assume that $ab=0$ for some $a, b \in R$. This implies $eab=0\Rightarrow eaeb=0$. This implies $(ea)er(eb) \in T(eR)$ for all $r \in R$. Thus $eaereb(es)^{n}=(es)^{n}eaereb \Rightarrow earbs^{n}=s^{n}earb$ for all $s \in R$. Thus we have  $ earb \in T(R)$.  Again  by applying similar procedure we  get $(1-e)arb \in T(R)$. Thus $earb+(1-e)arb= arb \in T(R)$.
\end{proof}
Upon examining Proposition \ref{p2}, it may seem that the requirement for $eR$ and $(1-e)R$ to be $\mathscr{H}$-semicommutative is superfluous. However, the example provided below will demonstrate that this is not the case.
 \begin{example}\label{e1}
 There exist a ring $R$ which is abelian but not $\mathscr{H}$-{\it semicommutative}.
 \end{example}
\begin{proof}
Let us consider a ring $R$ as follows:
\[R=\Bigg\{\begin{pmatrix}
u & v \\ 
w & x
\end{pmatrix}\,\,|\,\, u, v,w,x \in \mathbb{Z}, u\equiv x(mod 2), w \equiv v(mod 2) \Bigg\}\]. Then the only idempotent elements are $\begin{pmatrix}
1 & 0 \\ 
0 & 1
\end{pmatrix} $ and $\begin{pmatrix}
0 & 0 \\ 
0 & 0
\end{pmatrix}$. Hence it is abelian. Now let $P=\begin{pmatrix}
2 & 2 \\ 
0 & 0
\end{pmatrix}$ and $Q=\begin{pmatrix}
0 & 2 \\ 
0 & -2
\end{pmatrix}$. Thus $PQ=0$, let us take any arbitrary $C=\begin{pmatrix}
3 & 4 \\ 
0 & 1
\end{pmatrix}$ in $R$, then $PCQ=\begin{pmatrix}
0 & -8 \\ 
0 & 0
\end{pmatrix}$. We can see that $PCQ \notin T(R)$ as for $K=\begin{pmatrix}
4 & 0 \\ 
0 & 0
\end{pmatrix} \in R$ $PCQK^{n}=0$ but $K^{n}PCQ=\begin{pmatrix}
0 & -8.4^{n} \\ 
0 & 0
\end{pmatrix}$. Hence $PCQK^{n} \neq K^{n}PCQ$.
\end{proof}
Next, we investigate some triangular matrix subrings, which provide a good supply of rings that are $\mathscr{H}$-{\it semicommutative} but not {\it semicommutative}, as well as rings that are {\it nil-semicommutative} but not $\mathscr{H}$-{\it semicommutative}. For a ring $R$, let $T_{n}(R)$ denote the set of all $n \times n$ upper triangular matrices, and let $V_{n}(R)=\sum_{i=1}^{n-1}E_{i,i+1}$. For $T_{n}(R)$, consider the following subsets: 
\begin{center}
\[A_{n}(R)=\sum_{i=1}^{n}N(R)E_{ii}+\sum_{l=2}^{\floor{\frac{n}{2}}}RV_{n}^{l-1}+\sum_{i=1}^{\floor{\frac{n+1}{2}}}\sum_{j=\floor{\frac{n}{2}}+i}^{n}RE_{ij}\]\\
\[B_{n}(R)=\sum_{i=1}^{n}N(R)E_{ii}+\sum_{l=3}^{\floor{\frac{n}{2}}}RV_{n}^{l-2}+\sum_{i=1}^{\floor{\frac{n+1}{2}}+1}\sum_{j=\floor{\frac{n}{2}}+i-1}^{n}RE_{ij}\]\\
\[U_{n}(R)=\sum_{i=1}^{n}N(R)E_{ii}+\sum_{i=1}^{\floor{\frac{n-1}{2}}}\sum_{j=\floor{\frac{n}{2}}+1}RE_{ij}+\sum_{j=\floor{\frac{n-1}{2}}+2}^{n}RE_{\floor{\frac{n-1}{2}}+1,j}\]\\
\[S_{n}^{'}(R)=\sum_{i=1}^{n}N(R)E_{ii}+\sum_{i < j}RE_{ij}\]\\
\[S_{n}(R)=RI_{n}+\sum_{i < j}RE_{ij}\]\\
\[T(R,n)=RI_{n}+\sum_{l=2}^{n}RV_{n}^{l-1}\]
\[T^{'}(R,n)=\sum_{i=1}^{n}N(R)E_{ii}+\sum_{l=2}^{n}RV_{n}^{l-1}\]\\
\end{center}
\begin{proposition}
For a ring $R$,  the following statements are true.
\begin{itemize}
\item[(1)] If  $R$ is reduced, then $T(R, n), A_{n}(R), \, \, U_{n}(R)$ is central semicommutative.
\item[(2)] If  $R$  is nil simicommutative-II, then $T_{n}(R)$ is nil-semicommutative-II.
\item[(3)] For any ring $R$, $T_{n}(R)$ is not $\mathscr{H}$-semicommutative.
\item[(4)] If $R$ is $\mathscr{H}$-semicommutative, then  $S_{n}^{'}(R),\,\,     \,\, U_{n}(R), \, \, A_{n}(R),\,\, T_{n}^{'}(R,n)$ is $\mathscr{H}$-{\it semicommutative} but not central semicommutative.

\item[(5)]  If $R$ is $\mathscr{H}$-semicommutative, then $B_{n}(R)$ is $\mathscr{H}$-{\it semicommutative} for all $n \in \mathbb{N}$ but not central semicommutative for $n=2,3$.
\end{itemize}
\end{proposition}
\begin{proof}
\begin{itemize}
\item[(1)] The proof follows from Lemma 2.9 in \cite{TNA} and Theorem 2.1 in \cite{YG}.
\item[(2)] The  proof follows from Corollary 2.12 in \cite{YJ}.
\item[(3)] If possible, assume that for $n \geq 2$, $T_{n}(R)$ is $\mathscr{H}$-{\it semicommutative}. Then by \ref{p2}, $T_{n}(R)$ is an abelian ring. Consider $E_{22}$, which is clearly idempotent. Since $T_{n}(R)$ is abelian, this implies $E_{22} \in Z(T_{n}(R))$. However, we see that $E_{22}E_{12} = 0_{n \times n} \neq E_{12} = E_{12}E_{22}$. Hence, we get a contradiction, implying that $T_{n}(R)$ is not $\mathscr{H}$-{\it semicommutative}. 
\item[(4)]  Since $A_{n}(R), \,\,U_{n}(R), \,\,T^{'}(R, n)$ are subsets of $S^{'}_{n}(R)$. Thus it is sufficient to show that $S^{'}_{n}(R)$ is $\mathscr{H}$-{\it semicommutative}. Again since $R$ is {\it $\mathscr{H}$-semicommutative}, implies $nil(R)$ is an ideal of $R$ which makes $S^{'}_{n}(R)$ as a nil ring. Thus by Proposition \ref{p3},  $S^{'}_{n}(R)$ is  $\mathscr{H}$-{\it semicommutative}. Now we proceed to show that they are not central semicommutative ring. For $U_{n}(R)$, Let $a, c \in Nil(R)$ such that $a^{p}=0$ and $c^{q}=0$ for some $1 < p,q \in \mathbb{N}$. Then consider $A=\begin{pmatrix}
a & c^{q-1} \\ 
0 & c^{q-1}
\end{pmatrix} $ and $B=\begin{pmatrix}
a^{p-1} & a^{p-1} \\ 
0 & c
\end{pmatrix}$ be any elements of $U_{2}(R)$. Then we see that  $AB=\begin{pmatrix}
a & c^{q-1} \\ 
0 & c^{q-1}
\end{pmatrix}\begin{pmatrix}
a^{p-1} & a^{p-1} \\ 
0 & c
\end{pmatrix}=\begin{pmatrix}
0 & 0 \\ 
0 & 0
\end{pmatrix}$. Now consider $C=\begin{pmatrix}
0 & 1 \\ 
0 & 0
\end{pmatrix}$. Then  we have $ACB=\begin{pmatrix}
0 & ac \\ 
0 & o
\end{pmatrix}$. On the other hand we see that $\begin{pmatrix}
o & ac \\ 
0 & 0
\end{pmatrix}\begin{pmatrix}
t & 0 \\ 
0 & 0
\end{pmatrix}=\begin{pmatrix}
0 & 0 \\ 
0 & 0
\end{pmatrix}$ but however $\begin{pmatrix}
t & 0\\ 
0 & c
\end{pmatrix}\begin{pmatrix}
0 & ac \\ 
0 & 0
\end{pmatrix}=\begin{pmatrix}
0 & tac \\ 
0 & 0
\end{pmatrix}$ may not be a zero matrix. Thus $\begin{pmatrix}
0 & ac \\ 
0 & 0
\end{pmatrix} \notin  C(U_{2}(R))$. Hence $U_{2}(R)$ is not {\it central semicommutative}. Since $U_{2}(R)$ is embedded as subring in $U_{n}(R)$ for $n \geq 2$. Thus we can say that $U_{n}(R)$ is not {\it central semicommutative}. Similarly we can show that $S_{n}^{'}(R),\,\,  A_{n}(R),\,\, T_{n}^{'}(R,n)$ are not {\it central semicommutative}.
\item[(5).] Since $R$ is {\it $\mathscr{H}$-semicommutative}, implies $nil(R)$ is an ideal of $R$ which makes $B_{n}(R)$ as a nil ring. Thus by Proposition \ref{p3},  $B_{n}(R)$ is  $\mathscr{H}$-{\it semicommutative} for all $n\in \mathbb{N}$. On the other hand $B_{2}(R)$ is not {\it central semicommutative } by similar counter example given in (4).  
\end{itemize}
\end{proof}
Recall that if $uv=1$ implies $vu=1$ for any $u, v \in R$, then $R$ is called directly finite.
\begin{proposition}
Every $\mathscr{H}$-semicommutative ring is directly finite.
\end{proposition}
\begin{proof}
Let us consider $uv=1$ for some $u , \, v \in R$. This implies $u(vu-1)=0 \Rightarrow ur(vu-1) \in T(R)$ for all $r \in R$. This implies for some $n \in \mathbb{N}$, $v^{n}ur(vu-1)=ur(vu-1)v^{n}=vr(vuv^{n}-v^{n})=0$. Thus $v^{n}ur(vu-1)=0$. By taking $r=v$, we have $v^{n}uv(vu-1)=v^{n-1}(uv-1)=0$. Multiplying by $u$ from the left side we get $v^{n-2}(uv-1)=0$. Repeating the same left multiplication by $u$, after finite steps we get $vu=1$. 
\end{proof}
Recall that for a ring $R$, the intersection of all prime ideals is known as the prime radical, denoted by $P(R)$. Obviously, $P(R) \subseteq Nil(R)$. $R$ is called 2-primal if $P(R)=Nil(R)$. 
\begin{proposition}\label{p2.9}
Every $\mathscr{H}$-semicommutative ring is 2-primal ring.
\end{proposition}
\begin{proof}
We use induction on nilpotency of elements. Let $u \in R$ is such that $u^{2}=0$. Then by hypothesis we get $u r u \in T(R)$ for every $r \in R$. Thus for any $r, s \in R $, we have  $(us)^{n}uru=uru(us)^{n}=uru^{2}s(us)^{n-1}=0$, where $n=n(r,s) \in \mathbb{N}$. let $P$ be any prime ideal, then we have $us(us)^{n-1}uru \in P$ implies $u\in P$. Thus $u \in P(R)$. Again let us consider $u^{3}=0$, then this implies $uru^{2} \in T(R)$ for all $r \in R$. Then for some $ t \in \mathbb{N}$ we have $u^{t}uru^{2}=uruu^{t}=0$. Thus $usu^{t-1}uru^{2} \in T(R)$. Again by similar procedure its easy to see that $ u\in P(R)$. Thus by induction we conclude that $N(R)\subseteq P(R)$. 
\end{proof}

\begin{corollary}
Every $\mathscr{H}$-semicommutative rings is $J$-semicommutative.
\end{corollary}
We recall that for any $u \in Nil(R)$ and a module $_{R}M$, if  the left $R$-homomorphism $f : Ru \rightarrow\, _{R}M$ can be extended to $R\rightarrow\,  _{R}M$, then $_{R}M$ is called {\it left nil-injective}. On the other hand $_{R}M$ is called {\it left  wnil-injective} if for any $u \in nil(R)$, there exits a $n \in \mathbb{N}$ such that $u^{n}\neq 0$ and any left $R$-homomorphism $f : Ru^{n} \rightarrow\, _{R}M$ can be extended to $R\rightarrow\,  _{R}M$. It is obvious that {\it left nil-injective} implies {\it left wnil-injective}. Similarly  $R$ is called {\it left(right) nil-injective} if $_{R}R(R_{R})$ is {\it left(right) nil-injective} and $R$ is called {\it left(right) wnil-injective} if $_{R}R(R_{R})$ is {\it left(right) wnil-injective}.

\begin{proposition}\label{p2.21}
Let R be $\mathscr{H}$-semicommutative ring. If every simple singular left $R$-module is wnil-injective then $R$ is reduced.
\end{proposition}
\begin{proof}
Let us consider $u \in R$ such that $u^{2}=0$. If possible assume that $u \neq 0$. Then $l(u)\neq R$. Thus by hypothesis $\exists$ a maximal left ideal $K$ such that $u \in l(u) \subseteq K$. Again if $K$ is not essential in $_{R}R$ then $\exists$ some idempotent $e \in R$ such that $K=Re=R(1-e)$. This will imply that $u(1-e)=0$. Since $R$ is $\mathscr{H}$-semicommutative implies $(1-e)u=0$. Thus $(1-e) \in l(u) \subseteq l(1-e)$. Thus $(1-e)^{2}=0$ which is a contradiction. Thus $K$ is essential in $_{R}R$ and hence $R/K$ is $wnil-injective$. Now consider a module homomorphism  $f: Ru \rightarrow R/K$ given by $f(ru)=r+K$. Then by hypothesis $\exists$ $g : R \rightarrow R/K$ given as $g(u)=f(u)$. Now we have $1+k=f(u)=g(u)=ug(1)=uv+K$ where $g(1)=v+K$. Thus $1-uv \in K$. Also since $u^{2}=0$ and $R$ is $\mathscr{H}$-semicommutative, there exists some $t \in \mathbb{N}$ such that $(ru)^{t}(uru)=uru(ru)^{t}=0$. Thus we get $uv \in Nil(R)$ which implies $1-uv$ is unit in $R$. This gives a contradiction that $K$ is maximal. Hence wwe get $u=0$ and thus $R$ is reduced.
\end{proof}

A ring $R$ is referred to as a left (or right) $GP-V^{'}$-ring if every simple singular left (or right) $R$-module is GP-injective. A left ideal $L$ of $R$ is termed a $GW$-ideal if, for any $a \in L$, there exists a positive integer $n \geq 1$ such that $a^{n}R \subseteq L$. The concept of a $GW$-ideal for right ideals $K$ of $R$ is defined in a similar manner.

\begin{corollary}
Let R be $\mathscr{H}$-semicommutative ring. Then $R$ is strongly regular if and only if every maximal left ideal of $R$ are GW-ideals and $R$ is left GP-$V^{'}$ ring. 
\end{corollary}
\begin{proof}
Let us consider every maximal left ideal of $R$ are GW-ideals and $R$ is left GP-$V^{'}$ ring. Then by \ref{p2.21}, $R$ is reduced. Thus by Theorem \ref{p2.3}, $R$ is central semicommutative and Hence by Theorem 2.13 in \cite{LJ}, $R$ is strongly regular. The converse part is a well known result.
\end{proof}
\begin{corollary}
Let R be $\mathscr{H}$-semicommutative ring. Then $R$ is strongly regular if and only if every maximal right ideal of $R$ are GW-ideals and $R$ is right GP-$V^{'}$ ring. 
\end{corollary}
Recall that $R$ is called left SF-ring if every simple left $R$-module is flat.
\begin{proposition}\label{p2.14}
Let R be $\mathscr{H}$-semicommutative. If R is left SF, then $R$ is strongly regular
\end{proposition}
\begin{proof}
Let us consider $u \in R$ satisfies $uRu=0$. If possible assume that $u \neq 0$. Then there exits a  maximal left ideal $K$ such that 
$u \in r(uR) \subseteq K$. Again by hypothesis that $R$ is left SF, there exits a $t \in K$ such that $u=ut$. Thus $u(1-t)=0$ and since $R$ is $\mathscr{H}$, we have $uR(1-t) \subset T(R)$. Again if possible let us assume that $uR(1-t)\neq 0$. Then again by similar process we have a $g \in R$ such that $ug(1-t)\neq 0$. This implies $l(ug(1-t))\neq R$. Then again $\exists$ a maximal left ideal $L$ such that $l(ug(1-t)) \subseteq L$. Also  $uRu=0$ implies $ug(1-t)ug(1-t)=0$. This implies $ug(1-t) \in l(ug(1-t)) \subseteq L$. Thus there exits $l \in L$ such that $ug(1-t)=lug(1-t)\Rightarrow (1-l)(ug(1-t))=0\Rightarrow 1-l \in l(ug(1-t))\subseteq L$. Thus $1 \in L$ which gives us a contradictions. Thus $uR(1-t)=0\Rightarrow 1-t \in r(uR) \subseteq K\Rightarrow \, 1 \in K$ which is again a contradiction. This shows that $u=0$. Thus $R$ is semiprime. Hence by Proposition 2.3 $R$ is central semicommutative. Thus by Theorem 2.17 in \cite{LJ}, $R$ is strongly regular.
\end{proof}
 Recall that a right ideal $I$ of $R$ is called right essential if $I\cap J\neq 0$ for every non-zero right ideal $J$. Left essential ideals are defined similarly. The right singular ideal $\mathcal{Z}_{r}(R)$ is defined as the set of all $r \in R$ such that $R_{R}(r)$ is essential right ideal. left singular ideal $\mathcal{Z}_{l}(R)$ is defined similarly. It is well known that $\mathcal{Z}_{r}(R)$ and $\mathcal{Z}_{l}(R)$ are two sided ideals. In \cite{MA}, $R$ is called {\it nil singular } if both $\mathcal{Z}_{r}(R)$ and $\mathcal{Z}_{l}(R)$ are nil ideals. It is well known that if the  {\it A.C.C} condition is  satisfy on every left and right annihilator of $r \in R$, then $R$ is {\it nil singular}. Here we record some other  conditions for $R$ to be {\it nil singular}.
\begin{proposition}
Let $R$ be $\mathscr{H}$-semicommutative. Then $R$ is nil singular if any one of the following condition holds. 
\begin{itemize}
\item[(1).] $R/Nil(R)$ is left  SF-ring.
\item[(2).] $R/Nil(R)$ is left GP-$V^{'}$ ring with every maximal left  ideals  are GW-ideals.
\item[(3).] $R/Nil(R)$ is right GP-$V^{'}$ ring with every maximal right  ideals  are GW-ideals.
\end{itemize}
\end{proposition}
\begin{proof}
 $R$ is $\mathscr{H}$-semicommutative. Thus by proposition \ref{p2} $R$ is abelian and by  Proposition \ref{p2.9}, $R$ is 2-primal. This implies $Nil(R)$ is ideal and obviously $R/Nil(R)$ is reduced and hence $R/Nil(R)$ is $\mathscr{H}$-semicommutative. By Proposition \ref{p2.14} $R/Nil(R)$ is strongly regular. Hence $R/Nil(R)$ is regular. Thus $R$ is $\pi$-regular. Since every $\pi$-regular ring is nil-singular(see Proposition 2.2 in \cite{MA}). Thus result (1) follows. The proof for (2) and (3) follows similarly.
\end{proof}
Next, we study the localization of rings. Let $S$ denote a multiplicatively closed subset of $R$ consisting of centrally regular elements. Then $S^{-1}R$ is the localization of $R$ at $S$  

\begin{proposition}\label{p6}
Let R be a ring. Then $R$ is $\mathscr{H}$-semicommutative if and only if $S^{-1}R$ is $\mathscr{H}$-semicommutative.
\end{proposition}
\begin{proof}
Let us suppose that $R$ is $\mathscr{H}$-semicommutative. let $p_{1}=a^{-1}p $ and $q_{1}=b^{-1}q$ are any elements of $S^{-1}R$  are such that $p_{1}q_{1}=0$. Since  $a , b \, \in S$ and $S$ is central regular subset implies $p_{1}q_{1}=a^{-1}pb^{-1}q=a^{-1}b^{-1}pq=0$. This implies $pq=0$. So by hypothesis $prq \in T(R)$ for every $r \in R$. Let $u \in S$ and $t \in R$. Then we have $a^{-1}b^{-1}u^{-1} \in C(R)\subseteq T(R)$ and $ptq \in T(R)$. So $a^{-1}b^{-1}u^{-1}ptq=(a^{-1}p)(u^{-1}t)(b^{-1}q)=p_{1}t_{1}q_{1}\in T(R)$.  Thus $S^{-1}R$ is $\mathscr{H}$-{\it semicommutative}. Conversely let us suppose that $S^{-1}R$ is $\mathscr{H}$-{\it semicommutative}. Since $R$ is embedded as a subring in $S^{-1}R$ and  $\mathscr{H}$-{\it semicommutative} rings are closed under subring. Thus $R$  is  $\mathscr{H}$-{\it semicommutative}.
\end{proof}
\begin{corollary}
let $R$ be a ring. Then $R[x]$ is $\mathscr{H}$-{\it semicommutative} if and only if $R[x, x^{-1}]$ is $\mathscr{H}$-{\it semicommutative}.
\end{corollary}
\begin{proof}
Consider $S=\{1, x, x^{2}, \hdots\}$. Then  clearly $S$ is multiplicative closed subset of $R[x]$ consisting of central regular elements. Thus the proof follows form  Proposition \ref{p6}.
\end{proof}

Next we record an example of ring $R$ and an ideal $I$ of $R$ such that both $R$ and $I$ are $\mathscr{H}$-{\it semicommutaive} but the quotient ring $R/I$ is not $\mathscr{H}$-{\it semicommutaive}.
\begin{example}
Consider $R=\begin{pmatrix}
F & F & F \\ 
0 & F & F \\ 
0 & 0 & F
\end{pmatrix} $ and $I=\begin{pmatrix}
0 & F & F \\ 
0 & 0 & F \\ 
0 & 0 & 0
\end{pmatrix}$, where $F$ represent a field. Then $R/I\cong F\oplus F\oplus F$. Then its clear that $R/I$ is $\mathscr{H}$-{\it semicommutaive} and since $I \subseteq Nil(R)$ implies $I$ is also $\mathscr{H}$-{\it semicommutaive}. However $R$ is not $\mathscr{H}$-{\it semicommutaive}. For this let $A=\begin{pmatrix}
0 & 1 & -1 \\ 
0 & 0 & 0 \\ 
0 & 0 & 0
\end{pmatrix}$ and $B=\begin{pmatrix}
0 & 0 &  0\\ 
0 & 0 & 1 \\ 
0 & 0 & 1
\end{pmatrix}$. Then we have $A.B=0$. Let $C=\begin{pmatrix}
1 & 1 & 0 \\ 
0 & 1 & 0 \\ 
0 & 0 & c
\end{pmatrix}$. Then $ACB=\begin{pmatrix}
0 & 0 & 1 \\ 
0 & 0 & 0 \\ 
0 & 0 & 0
\end{pmatrix}$ Let us choose $X=\begin{pmatrix}
0 & 0 & 0 \\ 
0 & 0 & 0 \\ 
0 & 0 & c
\end{pmatrix}$ any arbitrary element of $R$ where $c \neq 0 $. Then we can see that for any $ n \in \mathbb{N}$ $X^{n}=\begin{pmatrix}
0& 0 & 0 \\ 
0 & 0 & 0 \\ 
0 & 0 & c^{n}
\end{pmatrix}$.  Thus $X^{n}ACB=0$ and $ACBX^{n}=\begin{pmatrix}
0 & 0 & c^{n} \\ 
0 & 0 & 0 \\ 
0 & 0 & 0
\end{pmatrix}$. Thus $X^{n}ACB \neq ACBX^{n}$ for any $n \in \mathbb{N}$. Hence $R$ is not $\mathscr{H}$-semicommutative.
\end{example} 
Here, we provide an example of a ring $R$ that is $\mathscr{H}$-{\it semicommutative}; however, $R/I$ is not $\mathscr{H}$-{\it semicommutative} for some ideal $I$ of $R$.

\begin{example}
Let $R=D[x,y]$ be a ring of polynomial over two variables where $D$ is a division ring. Let us consider $I=<x^{2}>$ be an ideal of $R$ such that $xy^{n}\neq y^{n}x$ for all $n \in \mathbb{N}$. Since $R$ is{\it  semicommutaive} implies $R$ is $\mathscr{H}$-{\it semicommutative} but, $(x+I)^{2}=I$ and $(x+I)(y^{n}+I)(x+I)$ is not in $T(R/I)$.
\end{example}
\begin{proposition}
Let   $R/I$  is $\mathscr{H}$-{\it semicommutative} for some ideal $I$ of $R$. If I is reduced then $R$  is $\mathscr{H}$-{\it semicommutative}. 
\end{proposition}
 \begin{proof}
 let $p, q \in R$ are such that $pq=0$. we have $qIa \subseteq I$ and $(qIa)^{2}=0$ and $qIa=0$ as $I$ is reduced. Therefore $((pRq)I)^{2}=0$ and so $(pRq)I=0$. Now $(p+I)(q+I)=pq+I=I$ and since $R/I$ is $\mathscr{H}$-{\it semicommutative} implies $(p+I)(R+I)(q+I)= pRq+I \in T(R/I)$. By hypothesis $(paq+I)(r^{n}+I)=(r^{n}+I)(paq+I)$. This implies $paqr^{n}-r^{n}paq \in I$ for every $a, r \in R$. So we get $(paqr^{n}-r^{n}paq)^{2}\in (paqr^{n}-r^{n}paq)I=0 $ since $(pRq)I=0$. Thus $paqr^{n}=r^{n}paq$ for each $r, a \in R$. Hence $R$ is $\mathscr{H}$-{\it semicommutative}.
 \end{proof}
 \begin{proposition}\label{p10}
 let $R$ be a Armendariz ring. Then the following conditions are equivalent:
 \begin{itemize}
 \item[(1)] $R$ is $\mathscr{H}$-semicommutative.
 \item[(2)] $R[x]$ is $\mathscr{H}$-semicommutative.
 \end{itemize}
 \end{proposition}
 \begin{proof}
 (2)$\Rightarrow$(1) The proof is straightforward since subring of $\mathbb{H}$-{\it semicommutative} ring are closed.(1)$\Rightarrow$(2) Let us consider $f(x)=\sum_{l=0}^{n}u_{l}x^{l}$ and $g(x)=\sum_{k=0}^{m}v_{k}x^{k}$ are any two elements of $R[x]$ such that $f(x)g(x)=0$. Since $R$ is Armendariz implies $u_{l}v_{k}=0$. Thus by hypothesis, $u_{l}Rv_{k} \in T(R)$. Since $T(R)$ is closed under addition, this implies $f(x)R[x]g(x) \in T(R[x])$. Hence $R$ is $\mathscr{H}$-{\it semicommutative}.
 \end{proof}
 Here next we have shown that the hypothesis of $R$ is Armendariz ring is not a superfluous.
\begin{example}
Let $L=\mathbb{Z}_{2}<d_{0}, d_{1}, d_{2}, d_{3}, e_{0}, e_{1}>$ be the free algebra(with 1) over $\mathbb{Z}_{2}$ generated by six indeterminates(as labeled above). Let $I$ be a ideal generated by  the following relations:
\begin{center}
$d_{0}e_{0}, d_{0}e_{1}+d_{1}e_{0}, d_{1}e_{1}+d_{2}e_{0}$,\\
 $d_{2}e_{1}+d_{3}e_{0}, d_{3}e_{1}, d_{0}a_{k}(0\leq k \leq 3), d_{3}d_{k}(0\leq k\leq 3)$,\\
 $ d_{1}d_{k}+d_{2}d_{k}(0\leq k \leq 3), e_{l}e_{k}(0\leq l, k\leq 1), e_{l}d_{k}(0\leq l \leq 1, 0\leq k \leq 3) $
\end{center}

Let $R=L/I$. Neilsen in \cite{PN}, proved that $A$ is  {\it semicommutative} but not Mcoy. Further, $R$ is not $\mathscr{H}$-{\it semicommutative}as consider $p(x)=d_{0}+d_{1}x+d_{2}x^{2}+d_{3}x^{3}$ and $m(x)=e_{0}+e_{1}x$. The above  relations  suggest that $p(x)m(x)=0$ in $R[x]$, but however we can see that $p(x)d_{0}m(x)\neq 0$ since $d_{1}d_{0}e_{1}+d_{0}d_{1}e_{2} \notin T(R[x]) $. 
\end{example}

Recall that an element $a \in R$ is said to be multiplicatively finite if the set $S = \{a^{t} \,|\, t \in \mathbb{N} \cup \{0\}\}$ is finite, and $k$ is said to be the multiplicative order of $a$ if $k$ is the least positive integer such that $a^{k}=a^{m}$ for some $0 \leq m < k$. Thus, idempotent elements are multiplicatively finite with a multiplicative order $k \leq 2$. We say a multiplicative element $a \in R$ satisfies property (P) if $a^{n}=a$ for some $ n \geq 2$.

Recall from \cite{BN}, $R$ is called left semi-Baer (right semi-quasi-Baer) if for any subset $X \subseteq R$ (left ideal $I \subseteq R$), we have $l_{R}(X)=Rb$ ($l_{R}(I)=Rb$) for some $b \in S$. A ring $R$ is called left semi-P.P ring (left semi p.q ring) if $l_{R}(a)=Rb$ ($l_{R}(Ra)=Rb$) for some $b \in S$. All these definitions have left-right symmetry. Thus, a ring $R$ is called semi-Baer (semi quasi-Baer) if it is both left and right semi-Baer (semi quasi-Baer). A ring $R$ is called semi P.P ring (semi p.q ring) if it is both left and right semi P.P ring (semi p.q ring).
\begin{proposition}\label{p2.23}
Let $R$ be an $\mathscr{H}$-{\it semicommutative} ring. If every multiplicative element satisfies property (P) and $R$ is a left (or right) semi-p.p ring, then $R$ is reduced.
\end{proposition}
\begin{proof}
Let $a \in R$ is such that $a^{2} = 0$. Thus $a \in l_{R}(a)$. Also  $R$ is semi-p.p ring so $\exists$ an idempotent $b \in S$ such that $l_{R}(a) = Rb$. By hypothesis $\exists \,\, n \in \mathbb{N}$ such that $b^{n}=b$. Thus we have $(b^{(n-1)})^{2}=b^{n-1}$ and since $\mathscr{H}$-{\it semicommutative} rings are abelian implies $b^{(n-1)}$ is idempotent. Since  $a=rb$ for some $r \in R$ and $b^{(2n-1)}=b^{n}.b^{(n-1)}=b.b^{(n-1)}=b^{n}=b$. Multiplying  $a=rb$ by $b^{(2n-2)}$ from right side we get $0=ab^{(2n-2)}=rb^{(2n-1)}=rb=a$. Thus $R$ is reduced. Similar argument also follows for right semi-p.p ring.
\end{proof}
\begin{proposition}
Let $R$ be an $\mathscr{H}$-{\it semicommutative} ring. If every multiplicative element satisfies property (P) and $R$ is a left (or right) semi-p.q ring, then $R$ is reduced."
\end{proposition}
\begin{proof}
Proof follows similar to that of Proposition \ref{p2.23}
\end{proof}
\begin{corollary}\label{p4}
Let $R$ be $\mathscr{H}$-{\it semicommutative}. Then $R$ is reduced if any one of the following conditions holds. 
\begin{itemize}
\item[(1)] R is semiprime.
\item[(2)] R is left(right) p.p ring.
\item[(3)] R is left(right) p.q-Baer ring.
\end{itemize}
\end{corollary}
\begin{proof}
The proof follows straightforwardly, as every idempotent element is a multiplicative finite element.
\end{proof}

\begin{corollary}
For a  $\mathscr{H}$-{\it semicommutative} ring $R$,  the following assertions are equivalent.
\begin{itemize}
\item[(1)] R is semi-left p.p ring.
\item[(2)] R is semi-right p.p ring.
\item[(3)] R is semi-left p.q-Baer ring.
\item[(4)] R is semi-right p.q-Baer ring.
\end{itemize}
\end{corollary}
\begin{proof}
The proof follows easily as in each of the cases $R$ is reduced.
\end{proof}
\begin{corollary}
For a   $\mathscr{H}$-{\it semicommutative} ring $R$, the following assertions  are equivalent:
\begin{itemize}
\item[(1)] R is left p.p ring.
\item[(2)] R is right p.p ring.
\item[(3)] R is left p.q-Baer ring.
\item[(4)] R is right p.q-Baer ring.
\end{itemize}
\end{corollary}
\begin{proof}
The proof follows easily as in each of the case $R$ is reduced.
\end{proof}
\begin{corollary}
For a  $\mathscr{H}$-{\it semicommutative} ring $R$, the  following assertions are equivalent:
 \begin{itemize}
 \item[(1)] R is semi-Baer ring.
 \item[(2)] R is semi quasi-Baer ring.
 \end{itemize}
 \end{corollary}
 \begin{proof}
 (1)$\Rightarrow$(2) is obvious. (2)$\Rightarrow$(1) Since If $R$ is quasi-Baer  ring which implies $R$ is semi p.p ring.  Thus by Proposition \ref{p4}, $R$ is reduced and hence {\it semicommutative}. Thus by Theorem 2.7 in \cite{NA} implies $R$ is Baer-ring.
 \end{proof}
 
 \begin{proposition}
For a  $\mathscr{H}$-{\it semicommutative} ring $R$, then  the following assertions are equivalent;
 \begin{itemize}
 \item[(1)] R is a right(left) p.p ring if and only if $R[x]$ is right(left) p.p ring.
 \item[(2)] R is a Baer ring if and only if $R[x]$ is a Baer ring.
 \item[(3)] R is right(left) p.q-Baer ring if and only if $R[x]$ is right(left) p.q-Baer ring.
 \item[(4)] R is a quasi-Baer ring if and only if $R[x]$ is a quasi-Baer ring.
 \end{itemize}
 \end{proposition}
 \begin{proposition}
 Let $R$ be a semi-p.p ring. If every multiplicative finite element satisfy the property (P), then the following assertions are equivalent:
 \begin{itemize}
 \item[(1)] $R$ is semicommutative ring.
 \item[(2)] $R$ is Armendariz ring.
 \end{itemize}
 \end{proposition}
 \begin{proof}
 (1)$\Rightarrow$(2) Since $R$ is {\it semicommutative} implies $R$ is reduced  and thus $R$ is Armendariz. (2)$\Rightarrow$(1) let us assume that $ab=0$ for some $a, \, b \in R$. This implies $b \in r_{R}(a)=cR$, where $c^{n}=c$. Then this implies $b=cr$. Again $c^{(n-1)^{2}}=c^{(n-1)}$. Thus $c^{n-1}$ is idempotent. Since $R$ is {\it semicommutative} implies $R$ is abelian,  multiplying  $b=cr$ by $c^{(2n-2)}$ from the left side we get $c^{(2n-2)}b=bc^{(2n-2)}=c^{(2n-1)}r=cr=b$. Now let $f(x)=ak+ax$ and $g(x)=kb-bx \in R[x]$.  We get $f(x)g(x)=(ak+ax)(kb-bx)=ak^{2}b=ak^{2}bc^{(2n-2)}=ak^{2}c^{(2n-2)}b=ac^{(2n-2)}k^{2}b=0$. By hypothesis $R$ is Armendariz, thus $akb=0$ for all $k \in R$. Hence $R$ is {\it semicommutative}.
 \end{proof}
 \begin{corollary}
For a p.p ring $R$, the following assertions are equivalent:
 \begin{itemize}
 \item[(1)] $R$ is semicommutative ring.
 \item[(2)] $R$ is Armendariz ring.
 \end{itemize}
 \end{corollary}
 \begin{corollary}
 Let $R$ be a semi-p.p ring. If every multiplicative finite element satisfy the property (P), then the following assertions are equivalent:
 \begin{itemize}
 \item[(1)] $R$ is $\mathscr{H}$-semicommutative ring.
 \item[(2)] $R$ is Armendariz ring.
 \end{itemize}
 \end{corollary}
 \begin{proposition}
 Let $R$ be  a $\mathscr{H}$-semicommutative. If every multiplicative finite element satisfy the property (P),  then the  following assertion  holds. 
 \begin{itemize}
 \item[(1)] If $R$ is semi-Baer ring then $R[x]$ is semi-Baer ring.
\item[(2)] If $R$ is semi-quasi Baer ring then $R[x]$ is semi quasi Baer ring.
\item[(3)] If $R$ is semi p.p ring then $R[x]$ is semi-p.p ring.
\item[(4)] If $R$ is semi-p.q ring then $R[x]$ is semi-p.q ring.
  \end{itemize}
 \end{proposition}
 \begin{proof}
 Let $R$ be semi-Baer ring, this implies $R$ is  semi-p.p ring. Hence by Proposition \ref{p4}, $R$ is reduced and this implies  $R[x]$ is reduced and hence $R[x]$ is semi-Baer ring. The other parts are direct implication of Corollary \ref{p4}
 \end{proof}

 \section{Conflict of Interest}
 On behalf of the authors the corresponding author declares no confict of interest.

\bibliographystyle{llltbibsty}
 \bibliography{h-semi}
\end{document}